\documentclass[a4paper, 11pt]{amsart}
\pdfoutput=1

\usepackage{amssymb}
\usepackage{amsfonts}
\usepackage{amsthm}
\usepackage{stmaryrd}
\usepackage[T1]{fontenc}
\usepackage{kpfonts} 
\usepackage[mathscr]{eucal}
\usepackage[english]{babel}
\usepackage[margin=1in]{geometry}
\usepackage{tikz-cd}
\usepackage{tikz}
\usetikzlibrary[trees]
\usepackage{enumitem}
\usepackage{hyperref}
\usepackage{aliascnt}
\usepackage{adjustbox}
\usepackage{varioref}
\usepackage{xstring}
\usepackage{xcolor}
\usepackage{bbm}

\newtheorem{thmintro}{Theorem}

\newcommand{\inewtheorem}[2]{
	\newaliascnt{#1}{thmintro}
	\newtheorem{#1}[#1]{#2}
	\aliascntresetthe{#1}
}
\inewtheorem{propintro}{Proposition}
\inewtheorem{corintro}{Corollary}

\newcommand{\jnewtheorem}[2]{
	\newaliascnt{#1}{theorem}
	\newtheorem{#1}[#1]{#2}
	\aliascntresetthe{#1}
}
\numberwithin{theorem}{section}
\jnewtheorem{lemma}{Lemma}
\jnewtheorem{proposition}{Proposition}
\jnewtheorem{corollary}{Corollary}
\jnewtheorem{conjecture}{Conjecture}

\theoremstyle{definition}
\jnewtheorem{definition}{Definition}
\jnewtheorem{example}{Example}
\jnewtheorem{remark}{Remark}
\jnewtheorem{question}{Question}
\jnewtheorem{fact}{Fact}
\jnewtheorem{construction}{Construction}
\jnewtheorem{notation}{Notation}
\jnewtheorem{convention}{Convention}

\labelformat{theorem}{Theorem #1}
\labelformat{lemma}{Lemma #1}
\labelformat{proposition}{Proposition #1}
\labelformat{corollary}{Corollary #1}
\labelformat{definition}{Definition #1}
\labelformat{example}{Example #1}
\labelformat{section}{Section #1}
\labelformat{subsection}{Subsection #1}
\labelformat{equation}{(#1)}
\labelformat{remark}{Remark #1}
\labelformat{exercise}{Exercise #1}
\labelformat{question}{Question #1}
\labelformat{chapter}{Chapter #1}
\labelformat{construction}{Construction #1}
\labelformat{conjecture}{Conjecture #1}
\labelformat{notation}{Notation #1}
\labelformat{convention}{Convention #1}

\labelformat{thmintro}{Theorem #1}
\labelformat{corintro}{Corollary #1}
\labelformat{propintro}{Corollary #1}

\numberwithin{table}{subsection}
\labelformat{table}{Table #1}

\hypersetup{
	pdftitle={},
	pdfauthor={Gregoire Marc},
	bookmarksnumbered=true,     
	bookmarksopen=true,         
	bookmarksopenlevel=1,       
	colorlinks=true,   
	linkcolor = {cyan!60!black},   
	urlcolor = {cyan!60!black},   
	citecolor = {cyan!60!black},   
	pdfencoding=unicode,      
	pdfstartview=Fit,           
	pdfpagemode=UseOutlines,    
	pdfpagelayout=OneColumn,
	breaklinks=true,
	linktocpage
}

\def\NN{{\mathbb{N}}}

\def\sP{{\mathscr{P}}}
\def\sS{{\mathscr{S}}}

\def\cA{{\mathcal{A}}}

\def\cC{{\mathcal{C}}}

\def\cF{{\mathcal{F}}}

\def\cO{{\mathcal{O}}}

\def\cP{{\mathcal{P}}}

\def\bigpp{\mathop{\mathchoice{\dobigpp\Huge}{\dobigpp\Large}{\dobigpp\normalsize}{\dobigpp\small}}}
\mathchardef\bigppchar"1403
\def\dobigpp#1{\vcenter{#1\kern.2ex\hbox{$\bigppchar$}\kern.2ex}}
\mathchardef\bigppdomchar"1400
\def\dobigppdom#1{\vcenter{#1\kern.2ex\hbox{$\bigppdomchar$}\kern.2ex}}

\mathchardef\pp"2403 
\mathchardef\ppdom"2400 

\newcommand\colim{\operatorname{colim}}

\newcommand\Id[1]{\operatorname{id}_{#1}}
\newcommand\Set{\operatorname{Set}}

\newcommand\fun{\operatorname{Fun}}

\newcommand\Arr{\operatorname{Arr}}

\newcommand\sSet{\operatorname{sSet}}
\newcommand\Algg{\operatorname{Alg}}

\newcommand\coll[1]{\operatorname{Coll}_{#1}}

\newcommand\Fin{\operatorname{Fin}}

\title{On sifted homotopy colimits of algebras over an $N_{\infty}$-operad}
\author{Gr\'{e}goire Marc}

\begin{document}
\maketitle

\begin{quote} 
\begin{center}
\textbf{Abstract} 
\end{center}
We prove that the forgetful functor from algebras over an $N_{\infty}$-operad to equivariant spaces preserves sifted homotopy colimits. In particular, we provide equivariant generalisations to results of Pavlov and Scholbach \cite{PS1} and Berger and Moerdijk \cite{BM03,BM07}.
\end{quote}

\tableofcontents

\section{Introduction}
Operads were introduced by May in \cite{May} to study infinite loop spaces, and they play a central role in the study of homotopical algebraic structures. The theory of operads and their algebras admits two main formalisms, namely $\infty$-operads \cite{HA} and topological operads \cite{May}. These two approaches, along with others such as dendroidal sets \cite{CisMoer}, have been compared in the work of Barwick \cite{Bar2}, Cisinski and Moerdijk \cite{cismo, CisMoer2}, and Chu, Haugseng, and Heuts \cite{CHH}.

The theories of $\infty$-operads and topological operads both offer different perspectives for homotopical algebra. The entire development of higher algebra by Lurie in \cite{HA} is based on $\infty$-operads, as they are more suited for categorical considerations. The importance of topological operads is more concrete and lies in their direct connection with topology and geometry. In particular, examples of such operads arise geometrically, such as with little discs operads which have a direct connection with configuration spaces \cite{May}.

The homotopy theory of algebras over a (topological) operad was developed and studied by Berger and Moerdijk in \cite{BM03,BM07}. An important player in this study is the notion of \emph{$\Sigma$-cofibrant operads}. The $\Sigma$-cofibrancy of a given operad $\cO$ is a technical condition that ensures the good behavior of the homotopy theory of its algebras. One important consequence of this is the rectification result of Pavlov and Scholbach \cite{PS1}, which states that the homotopy theory of algebras over a $\Sigma$-cofibrant operad can be formally compared with the algebras over the associated $\infty$-operad. This applies in particular to $E_{\infty}$-operads, which are encoding infinite loop spaces.

Let $G$ be a finite group. The theory of $N_{\infty}$-operads was introduced by Blumberg and Hill in \cite{BH1} as an equivariant generalisation of the theory of $E_{\infty}$-operads. Unlike in the $E_{\infty}$ case,
$N_{\infty}$-operads do not form a contractible space; rather, they assemble into a poset, which can be described in terms of \emph{$G$-indexing systems}, by work of Blumberg and Hill \cite{BH1}, Bonventre and Pereira \cite{BonPer1}, and Rubin \cite{Rub1}. More generally, one can consider equivariant topological operads (and their algebras), which provides interesting examples, such as equivariant little discs operads \cite{Hill2022}. As in the non-equivariant setting, there is also a notion of $G$-$\infty$-operads \cite{DJ} along with  nerve functors \cite{Bon} relating equivariant topological operads and their algebras with their $\infty$-categorical analogues.

However, these comparison functors are not yet known to be equivalences, even under some assumptions on the operads. In particular, $N_{\infty}$-operads do not yet admit a rectification result for their algebras. A first problem that arises when one attempts to prove such a rectification for $N_{\infty}$-operads is that these operads are, in general, not $\Sigma$-cofibrant. This implies in particular that all the results of \cite{BM03} and \cite{PS1} cannot be applied in the proof of a potential rectification. The goal of this paper is to generalise certain results of \cite{BM03} and \cite{PS1} to a broader class of equivariant operads which contains $N_{\infty}$-operads. More precisely, we prove the following theorem.

\begin{thmintro}[\protect{\ref{nsifco}}]\label{mainTH}
Let $\cO$ be a $G$-simplicial operad. We denote by $\sS^G$ the $\infty$-category of equivariant spaces, by $\sSet^G$ the category of $G$-simplicial sets, and by $W_G$ the class of morphisms of $\cO$-algebras which forget to $G$-weak equivalences.
If for every $n\ge 0$, $\cO_n$ has a free $\Sigma_n$-action, then the forgetful functor
$$
U\colon \Algg_{\cO}\left(\sSet^G\right)\left[W_G^{-1}\right] \to \sS^G
$$
preserves sifted colimits.
\end{thmintro}

\ref{mainTH} is a generalisation of \cite[Theorem 7.9]{PS1}. Note that \cite[Theorem 7.9]{PS1} can also be applied to equivariant operads but not $N_{\infty}$-operads, as they are not $\Sigma$-cofibrant (see \ref{Ninfnotscof}). In particular, the proof of \ref{mainTH} differs from the classical proof of \cite[Theorem 7.9]{PS1} and relies heavily on equivariant techniques. As we will see in \ref{cex}, some important results of \cite{BM03} do not extend to $N_{\infty}$-operads, which complicates the proof of \ref{mainTH}.

The result of this paper is used in \cite{greg1} to prove a rectification for algebras over $N_{\infty}$-operads, thus correcting the original version of this paper, which contained a mistake. Although \ref{mainTH} should be a first step towards a more general rectification for algebras over $G$-operads, the author does not know how to obtain such a rectification in full generality. 

\hspace{2cm}

\subsection*{Acknowledgements}
\addtocontents{toc}{\protect\setcounter{tocdepth}{1}}
I would like to thank my PhD supervisor, Magdalena Kędziorek, for her continuous support and for her valuable comments on earlier drafts of this paper. This work was supported by EPSRC grant EP/Z000580/1. During the writing of this paper, the author was supported by an NWO grant Vidi.203.004.

\section{Counter example}\label{cex}
In this section, we enlighten the fact that the results of \cite{BM03} cannot be genralised to $N_{\infty}$-operads. In particular, after proving that $N_{\infty}$-operads are not $\Sigma$-cofibrant in the sense of \cite{BM03}, we show that \cite[Corollary 5.5]{BM03} is false for $N_{\infty}$-operads.
Let $I$ be a small category and consider the equivariant model structure on $\fun(I,\sSet^G)$ obtained from the projective model structure on $\fun(I,\sSet)$. Note that this model structure exists using \cite[Lemma 2.12]{STE}.

\begin{proposition}\label{Ninfnotscof}
Let $\cO$ be an $N_{\infty}$-operad in $\sSet^G$. If the $G$-indexing system associated with $\cO$ is not trivial, then $\cO$ is not $\Sigma$-cofibrant.
\begin{proof}
The projective model structure on $\left(\sSet^G\right)^{\Sigma_n}$ corresponds to the model structure on $\sSet ^{G\times \Sigma_n}$ given by the family $\cF$ of $G\times \Sigma_n$ spanned by the subgroups $H\times \{e\}$, with $H$ a subgroup of $G$. If $\cO$ is not a trivial $N_{\infty}$-operads, there exists a graph subgroup $\Gamma \subseteq G\times \Sigma_n$ for some $n$ such that $\Gamma \notin \cF$ and such that $\cO_n^{\Gamma} \neq \emptyset$.  Using \cite[Proposition 2.16]{STE}, we deduce that $\cO_n$ is not cofibrant in $\sSet^{G\times \Sigma_n}$ for the model structure corresponding to $\cF$. In particular, this implies that $\cO$ is not $\Sigma$-cofibrant.
\end{proof}
\end{proposition}

\begin{proposition}\label{descof}
For every cellular map $f$ in $\fun (I,\sSet^G)$, the map $f(n) \in \fun(I,\Set^G)$ is of the form
\begin{equation}\label{form}
X \to X \sqcup \left(\bigsqcup_{j\in J} G/H_j \times \hom(A_j,(-))\right).
\end{equation}
In particular, cellular objects are of the form
$$
 \bigsqcup_{j\in J} G/H_j \times \hom(A_j,(-)).
$$
\begin{proof}
Sequential colimits of maps of the form \ref{form} also have this form and it is enough to show that for every pushout square
\begin{equation}\label{sqf}
\begin{tikzcd}
G/H \times \hom(i,(-))\times \partial \Delta^n \arrow[r] \arrow[d] &  X \arrow[d,"f"] \\
G/H \times \hom(i,(-)) \arrow[r] \times \Delta^n & Y,
\end{tikzcd}
\end{equation}
the map $f(n)$ is of the form \ref{form}. If we specify the square \ref{sqf} at some simplicial degree $k$, we obtain a pushout square
$$
\begin{tikzcd}
\bigsqcup_{\partial \Delta^n (k)} G/H \times \hom(i,(-))  \arrow[r] \arrow[d] &  X(k) \arrow[d,"f"] \\
\bigsqcup_{\Delta^n (k)} G/H \times \hom(i,(-)) \arrow[r]  & Y(k)
\end{tikzcd}
$$
and because the map of sets $\partial \Delta^n (k) \hookrightarrow \Delta^n(k)$ is an inclusion, we deduce that $Y(k)$ is the coproduct 
$$
X(k) \sqcup \left( \bigsqcup_{\Delta^n (k)\backslash \partial \Delta^n (k) }G/H \times \hom(i,(-))\right)
$$
and that $f(k)$ corresponds to the associated summand inclusion.
\end{proof}
\end{proposition}

\begin{proposition}\label{stabcof}
For every cofibrant $X$ in $\fun (I,\sSet^G)$, for every $f\colon i \to j$ in $I$ and every $n$, $X(f)(n) \colon X(i)(n) \to Y(i)(n)$ preserves the stabilisers of elements.
\begin{proof}
Note first that every retract of an object $X$ in $\fun (I,\sSet^G)$ that satisfies the condition from the statement also satisfies this condition. The result now follows from the fact that cellular objects clearly satisfy the condition from the statement using \ref{descof}.
\end{proof} 
\end{proposition}

\begin{proposition}\label{countexninf}
Let $\cO$ be an $N_{\infty}$-operad for $C_2$ with a complete indexing system. The forgetful functor $U\colon \Algg_{\cO}\left(\fun(\Fin, \sSet)^{C_2} \right) \to \fun(\Fin, \sSet)^{C_2}$ does not preserve cofibrant objects.
\begin{proof}
Consider the graph subgroup $\Gamma \subseteq C_2 \times \Sigma_2$ given by the diagonal. The simplicial set $\cO_2^\Gamma$ is contractible and, in particular, there is a $\Gamma$-fixed point in the set $\cO_2(0)$. Because the $\Sigma_2$-action on $\cO_2$ is free, we obtain that the corresponding $C_2 \times \Sigma_2$-equivariant map $(C_2 \times \Sigma_2)/\Gamma \to \cO_2(0)$ is injective. Define $X$ as the free $\cO$-algebra on $\hom(*,(-))$ seen as a discrete simplicial presheaf. We obtain a $C_2$-monomorphism
$$
\hom(\{1,2\},(-))\simeq (C_2 \times \Sigma_2)/\Gamma\times_{\Sigma_2} \hom(*,(-))^2 \hookrightarrow \cO_2(0) \times_{\Sigma_2} \hom(*,(-))^2
$$
where the last term is a $C_2$-subset of $U(X)_0$ and where the $C_2$-action on $\hom(\{1,2\},(-))$ is induced by the free $C_2$-action on $\{1,2\}$. The identity map $\Id{\{1,2\}}$ seen as an element of $\hom(\{1,2\},\{1,2\})$ has the trivial group as stabiliser and the unique map $\{1,2\} \to *$ seen as an element of $\hom(\{1,2\},*)$ has $C_2$ as stabiliser. This contradicts \ref{stabcof} and it follows, as $\hom(\{1,2\},(-))$ is a $C_2$-subobject of $U(X)_0$, that $U(X)$ is not cofibrant. 
\end{proof}
\end{proposition}

\section{Cellular extensions of strong fixed cofibrant operads}\label{celexop}

Until the end of this paper, $I$ is a small sifted category with finite coproducts. The category $\fun(I,\sSet)$ will always be endowed with the projective model structure.

\begin{definition}\label{defgraphcof}
For every $n\ge 0$, the \emph{graph model structure} on $\fun(I,\sSet)^{G\times \Sigma_n}$ is the model structure induced by the projective model structure on $\fun (I,\sSet)$ and associated with the family of graph subgroups of $G\times \Sigma_n$, which exists using \cite[Lemma 2.12]{STE}. We also call the corresponding model structure on collections in $\fun(I,\sSet)^G$ the graph model structure. An operad $\cO$ in $\fun(I,\sSet^G)$ is \emph{graph $\Sigma$-cofibrant} if its identity map $\eta_{\cO}\colon 1 \to \cO$ is a graph cofibration of collections.
\end{definition}

\begin{remark}
Note that because $I$ admits coproducts, $\fun(I,\sSet)$ is a \newline monoidal model category.   
\end{remark}

\begin{definition}
A morphism $f\colon A \to B$ in $\fun(I,\sSet)^G$ is a \emph{fixed cofibration} if, for every subgroup $H$ of $G$, the following conditions are satisfied:
\begin{enumerate}
    \item the morphism $f^H \colon A^H \to B^H$ is a cofibration in $\fun(I,\sSet)$;
    \item the colimit of $f$, seen as a functor $I \to \Arr(\sSet)^G$, commutes with $H$-fixed points.
\end{enumerate}
\end{definition}

\begin{proposition}\label{psrainy}
The class of fixed cofibrations is closed under pushouts and transfinite compositions.
\begin{proof}
We prove first the result for transfinite compositions. 
The point (1) is a direct consequence of the fact that fixed points commute with sequential colimits and of the fact that transfinite compositions of cofibrations are cofibrations. For the point (2), this follows from the fact that colimits commute with colimits and from the fact that fixed points commute with sequential colimits. For pushouts, the proof is similar and uses the fact that the underlying map (if we forget the $G$-action) of a fixed cofibration is a monomorphism and the fact that fixed points commute with pushouts along monomorphisms in any presheaf category.
\end{proof}
\end{proposition}

\begin{proposition}
Every cellular map in $\fun(I,\sSet)^G$ (for the model structure of \ref{defgraphcof}) is a fixed cofibration. 
\begin{proof}
Using \ref{psrainy}, it is enough to prove the result for generating cofibrations. Consider $f\times G/H$ with $f$ a cofibration in $\fun(I,\sSet)$. We have $(G/H \times f)^K\simeq f\times (G/H)^K$ which is a cofibration as $f$ is a cofibration. Moreover, $\colim(G/H \times f)^K\simeq (G/H)^K\times \colim f \simeq \colim ((G/H)^K\times  f)$.
\end{proof}
\end{proposition}

\begin{notation}
For every partition $\sP=(U_1,\ldots,U_k)$ of $\{1,\ldots,n\}$, we denote by $\Sigma_{\sP}$ the subgroup of $\Sigma_n$ spanned by the permutations that preserve $\sP$.
\end{notation}

\begin{definition}
A \emph{marked partitioned tree} is a marked tree $T$ (in the sense of \cite[Section 5]{PS1}) with a partition $\sP$ of the set of its leaves.
\end{definition}

For a given marked partitioned tree $T$ with partition $\sP$, we denote by $\operatorname{Aut}_{\sP}(T)$ the subgroup of $\operatorname{Aut}(T)$ spanned by the automorphisms of $T$ that preserve $\sP$.

\begin{remark}
The subgroup $\operatorname{Aut}_{\sP}(T)$ of $\operatorname{Aut}(T)$ is the inverse image of $\Sigma_{\sP}$ by the homomorphism $\alpha\colon \operatorname{Aut}(T) \to \Sigma_n$ that sends an automorphism of $T$ to the associated permutation of its leaves.
\end{remark}

\begin{construction}\label{decomptree}
Consider a partitioned marked tree $T$ with partition $\sP$ and suppose we have the grafting operation $t_m(T_1,\ldots,T_m)$ representing $T$. We can endow every $T_i$ with a partition $\sP_i$ induced $\sP$. As in \cite[Lemma 5.9]{BM03}, we can partition the set $\{T_1,\ldots,T_m \}$ into subsets of pairwise isomorphic partitioned marked trees $\{T^1_{1},\ldots,T^1_{m_1} \}\cup \cdots \cup \{T^k_{1},\ldots,T^1_{m_k} \}$ with $m_1 + \cdots + m_k =m$. We finally obtain that 
$$
\operatorname{Aut}_{\sP}(T)\simeq \left(\operatorname{Aut}_{\sP_1}(T_1)^{m_1}\times \cdots \times  \operatorname{Aut}_{\sP_k}(T_k)^{m_k} \right) \rtimes \left( \Sigma_{m_1}\times \cdots \times \Sigma_{m_k} \right).
$$
\end{construction}

\begin{notation}
Given a partition $\sP=(U_1,\ldots,U_k)$ of $\{1,\ldots, n\}$ with $|U_i|=n_i$ and a tuple $g=(g_1,\ldots,g_k)$ with $g_i\colon A_i \to B_i$ a map in $\fun (I,\sSet)^G$, we denote by $\bigpp_{\sP}f$ the $\Sigma_{\sP}$-object
$$
\bigpp_{\sP}g:=\bigpp_{i=1}^k g_i^{\square^{n_i}},
$$
where $\square$ denotes the pushout-product.
\end{notation}

\begin{definition}
A morphism $f\colon A \to B$ in $\fun(I,\sSet)^{G\times \Sigma_n}$ is a \emph{strong fixed cofibration} if for every partition $\sP=(U_1,\ldots,U_k)$ of $\{1,\ldots,n\}$ with $|U_i|=n_i$ and for every tuple $g=(g_1,\ldots,g_k)$ of fixed cofibrations, the morphism 
$$
\left(f\square \bigpp_{\sP}  g\right)_{\Sigma_{\sP}}
$$
is a fixed cofibration, where $\Sigma_{\sP}$ acts on $f$ by restriction of the $\Sigma_n$-action. A morphism $f\colon A \to B$ of collections in $\fun(I,\sSet)^G$ is a \emph{strong fixed cofibration} if $f(n)\colon A(n) \to B(n)$ is a strong fixed cofibration for every $n$ in $\mathbb{N}$.
An operad $\cO$ in $\fun(I,\sSet)^G$ is a \emph{strong fixed cofibrant operad} if the morphism of collections $\eta_{\cO}\colon 1 \to \cO$ provided by the identity of $\cO$ is a strong fixed cofibration.
\end{definition}

\begin{proposition}\label{staco}
The class of strong fixed cofibrations is closed under pushouts and transfinite compositions.
\begin{proof}
We prove first the result for pushouts. Consider a strong fixed cofibration $f\colon A \to B$ and a pushout $g\colon C \to D$ of $f$. Using \cite[Proposition 3.1.5]{PS1}, we obtain that $g \square \bigpp_{\sP}h$ is a pushout of $f \square \bigpp_{\sP}h$ and, in particular, $\left(g \square \bigpp_{\sP}h\right)_{\Sigma_{\sP}}$ is a pushout of $\left(f \square \bigpp_{\sP}h\right)_{\Sigma_{\sP}}$. The result then follows from \ref{psrainy}. We prove now the result for transfinite compositions. Note first that, using \cite[Proposition 3.1.6]{PS1}, strong fixed cofibrations are closed under finite compositions. Suppose that $f\colon A_0 \to A$ is the transfinite composition of 
$$
A_0 \overset{f_0}{\longrightarrow} A_1 \overset{f_1}{\longrightarrow}A_2 \overset{f_2}{\longrightarrow} \cdots
$$
where each $f_i$ is a strong fixed cofibration. We denote by $F_i\colon A_0 \to A_{i+1}$ the composition of 
$$
A_0 \overset{f_0}{\longrightarrow} A_1 \overset{f_1}{\longrightarrow} \cdots \overset{f_i}{\longrightarrow} A_{i+1}.
$$
As every $f_i$ is a strong fixed cofibration, every $F_i$ is also a strong fixed cofibration. Finally, the pushout-product $f\square \bigpp_{\sP} h$ is the sequential colimit 
$$
\colim_{i \in \NN} \left( F_i \square \bigpp_{\sP} h \right)
$$
and in particular $\left(f\square \bigpp_{\sP} h\right)_{\Sigma_{\sP}}$ is the colimit
$$
\colim_{i \in \NN}\left( \left( F_i \square \bigpp_{\sP} h \right)_{\Sigma_{\sP}}\right).
$$
As the colimit of every $(F_i \square \bigpp_{\sP} h)_{\Sigma_{\sP}}$ commutes with fixed points, we deduce that the colimit of $\left(f\square \bigpp_{\sP} h\right)_{\Sigma_{\sP}}$ also commutes with fixed points using that sequential colimits commute with fixed points.
We now prove that the sequential diagram $i \mapsto \left( F_i \square \bigpp_{\sP} h \right)_{\Sigma_{\sP}}^H$ is a cofibration in $\fun (\NN, \fun(I,\sSet))$, where $\NN$ is seen as a category via its poset structure. In particular, this will imply that $\left(f\square \bigpp_{\sP} h\right)_{\Sigma_{\sP}}^H$ is a cofibration. We denote respectively by $B$ and $C$ the source and target of $\bigpp_{\sP}h$. Using the classical description of cofibrations in $\fun (\NN, \fun(I,\sSet))$ in terms of the Reedy model structure, we have to prove that the morphism
\begin{equation}\label{mor}
\left(
\left(A_n \times C \right) \cup_{\left(A_n \times B \cup_{A_0\times B}A_0 \times C  \right)} \left(A_{n+1}\times B \cup_{A_0 \times B} A_0 \times C \right) \right)_{\Sigma_{\sP}}^H \to \left( A_{n+1}\times C \right)_{\Sigma_{\sP}}^H
\end{equation}

induced by the square
$$
\begin{tikzcd}
\left(A_n \times B \cup_{A_0\times B}A_0\times C \right)_{\Sigma_{\sP}}^H  \arrow[r] \arrow[d,"\left(F_n \square \bigpp_{\sP}h\right)_{\Sigma_{\sP}}^H"'] & \left( A_{n+1}\times B \cup_{A_0 \times B} A_0 \times C \right)_{\Sigma_{\sP}}^H \arrow[d,"\left(F_{n+1}\square \bigpp_{\sP}h\right)_{\Sigma_{\sP}}^H"] \\
\left( A_n \times C \right)_{\Sigma_{\sP}}^H \arrow[r,""'] & \left( A_{n+1}\times C \right)_{\Sigma_{\sP}}^H
\end{tikzcd}
$$
is a cofibration. Using that pushouts commute we pushouts, the morphism \ref{mor} can be identified with $\left( f_n \square \bigpp_{\sP}h \right)_{\Sigma_{\sP}}^H$ and the result follows from the fact that $f_n$ is a strong fixed cofibration.
\end{proof}
\end{proposition}

\begin{proposition}\label{twistedpushprod}
For every partition $\sP=(U_1,\ldots,U_k)$, for every tuple $g=(f_1,\ldots,f_k)$ of fixed cofibrations and for every group homomorphism $\alpha\colon H \to \Sigma_{\sP}$, the map
$$
\alpha^* \left(\bigpp_{\sP} f\right)
$$
is a fixed cofibration. 
\begin{proof}
The homomorphism $\alpha\colon H \to \Sigma_{\sP}$ gives an $H$-set $T$, which we can decompose $T$ into orbits as $T\simeq H/K_1 \sqcup \cdots \sqcup H/K_l$. In particular, the homomorphism $\alpha$ factors as a homomorphism $\alpha\colon H \to \Sigma_{[H:K_1]}\times \cdots \times \Sigma_{[H:K_l]}$ and we denote by $\alpha_i \colon H \to \Sigma_{[H:K_i]}$ the associated homomorphism. Using the fact that fixed points commute with products and that pushouts along a monomorphism commute with fixed points, we deduce that 
$$
\alpha^* \left(\bigpp_{\sP} f\right)^H \simeq \bigpp_{i=1}^l \alpha_i^* \left(f_i^{\square^{[H:K_i]}} \right)^H
$$
with $f_i$ the component of $f$ for which $H/K_i \subset U_i$.
If we assume that for every $i=1,\ldots, l$, $ \alpha_i^* \left(\bigpp_{\sP_i} f\right)$ is a fixed cofibration, we obtain that 
 $\alpha^* \left(\bigpp_{\sP} f\right)$ is also a fixed cofibration using that pushout-products of cofibrations are cofibrations and the fact that the functor 
$$
\square \colon \Arr\left(\fun(I,\sSet)^G\right)^2 \to \Arr\left(\fun(I,\sSet)^G\right)
$$
preserves sifted colimits.
We can now asumme that $T=H/K$ and that $\alpha$ is a homomorphism $\alpha\colon H \to \Sigma_{n}$ with $n=[H:K]$. We now prove that
$$
\alpha^* \left( f^{\square^n}\right)^H\simeq f^K.
$$
If we write $f\colon A \to B$ as an inclusion, the iterated pushout-product $f^{\square^n}$ can be described as the inclusion map
$$
f^{\square^n}\colon \bigcup_{i=1}^n B^{i-1} \times A \times B^{n-i} \hookrightarrow B^n.
$$
We prove first that $\alpha^*(B^n)^H\simeq B^K$. Consider the map $p\colon B^K \to B^n$ that sends a $K$-fixed point $x$ to the element $p(x)=(h_1x,\ldots,h_nx)$ with $e_H=h_1,\ldots,h_n$ representatives in $H$ of the cosets in $H/K$ with the ordering given by the chosen isomorphism $G/H \simeq \{1,\ldots,n\}$. The element $p(x)$ is an $H$-fixed point, moreover, $p$ induces an isomorphism $\alpha^*(B^n)^H\simeq B^K$ where the inverse of $p$ is given by the projection $\alpha^*(B^n)^H \to B^K$ to the first component. Finally, using a similar argument, we have that 
$$
\alpha^* \left(\bigcup_{i=1}^n B^{i-1} \times A \times B^{n-i} \right)^H\simeq  A^K
$$
and that the map $\alpha^*\left(f^{\square^n}\right)^H$ corresponds to the inclusion $f^K \colon A^K \hookrightarrow B^K$. In particular, we obtain that $\alpha^*(f^{\square^n})^H$ is a cofibration as $f$ is a fixed cofibration. Moreover, we have that
\begin{align*}
& \colim \left( \alpha^*\left(f^{\square^{n}}\right)^H  \right) \\
\simeq & \colim \left(f^K \right)\simeq (\colim f)^K 
\simeq   \left( \alpha^*\left((\colim f)^{\square^{n}}\right)  \right)^H\simeq \left( \colim  \alpha^*\left(f^{\square^{n}}\right)  \right)^H
\end{align*}
and it follows that $\alpha^*\left(f^{\square^{n}}\right)$ is a fixed cofibration.
\end{proof}
\end{proposition}

\begin{construction}\label{consquo}
Let $\Theta$ be a finite group, let $X$ be a $(G\times \Theta)$-set with a free $\Theta$-action and let $Y$ be in $\fun(I,\sSet)^{G\times \Theta}$. We will now describe the quotient $\left( X\times Y\right)_{\Theta}$ as a $G$-object in $\fun(I,\sSet)$. Choose first a set $R$ of representatives in $X$ of the classes in $X_{\Theta}$. For every representative $x$ and every $g$ in $G$, there is a unique representative $y$ and a unique $\theta \in \Theta$ such that $gx=\theta y$. The uniqueness of $\theta$ is given by the fact that the $\Theta$-action on $X$ is free. We now construct a $G$-action on $X_{\Theta}\times Y$. For every $g$ in $G$, the associated map $g\cdot(-)\colon X_{\Theta}\times Y \to X_{\Theta}\times Y$ sends the $a$-copy of $Y$ to the $(g\cdot a)$-copy of $Y$ and is given on $Y$ by the map $(\theta,g)\cdot(-)$ where $\theta$ is the unique element of $\Theta$ such that $g\cdot x = \theta \cdot y$ with $x$ and $y$ the representatives of $a$ and $g\cdot a$ respectively. There is a $G$-equivariant map $X\times Y \to X_{\Theta}\times Y$ that sends $(a,y)$ to $(x,\theta^{-1} y)$ with $x$ the representative of $a$ and $\theta$ the unique element of $\Theta$ such that $\theta x=a$. This map exhibits $X_{\Theta}\times Y$ with the $G$-action that we described above as the quotient $(X\times Y)_{\Theta}$. Let $H$ be a subgroup of $G$. We now describe the $H$-fixed points of $\left( X \times Y\right)_{\Theta}$. For every $H$-fixed point $a$ in $X_{\Theta}$ represented by $x$, the map $\alpha_{x}\colon H \to \Theta$ that sends an element $h$ to the unique $\theta$ in $\Theta$ such that $hx=\theta x$ is a homomorphism. We now obtain that the $H$-fixed points of $(X \times Y)_{\Theta}$ is given by
$$
\bigsqcup_{x \in X_{\Theta}} \alpha_x^*\left(Y \right)^H.
$$
Note that this construction also applies to maps in $\fun(I,\sSet)^{G\times \Theta}$.
\end{construction}

\begin{proposition}\label{propcelfixcof}
Every cellular cofibration in $\fun(I,\sSet)^{G\times \Sigma_n}$ for the graph model structure is a strong fixed cofibration.
\begin{proof}
Using \ref{staco}, it is enough to prove the result for generating cofibrations.
Consider the cofibration $f=T \times f'$ with $T$ a finite $G\times \Sigma_n$-set with a free $\Sigma_n$-action and $f' \colon A \to B$ a cofibration in $\fun(I,\sSet)$. Using \ref{consquo}, we have that $\left((f\square \bigpp_{\sP} g)_{\Sigma_{\sP}}\right)^H$ is the map
$$
\bigsqcup_{x\in T_{\Sigma_{\sP}}^H} f'\square \alpha_{x}^* \left(\bigpp_{\sP} g\right)^H 
$$
which is a cofibration as a coproduct of pushout-products of cofibrations using \ref{twistedpushprod}. We now have to prove that the colimit of $(f\square \bigpp_{\sP} g)_{\Sigma_{\sP}}$ commutes with $H$-fixed points. We have that
\begin{align*}
&\colim \left( \bigsqcup_{x\in T_{\Sigma_{\sP}}^H} f'\square \alpha_{x}^* \left(\bigpp_{\sP} g\right)^H  \right)  \\
\simeq& \bigsqcup_{x\in T_{\Sigma_{\sP}}^H} \colim(f')\square \left(  \alpha_{x}^* \left(\bigpp_{\sP} \colim g\right)\right)^H \simeq \left(\colim \left(f\square \bigpp_{\sP} g\right)_{\Sigma_{\sP}}\right)^H
\end{align*}

using \ref{twistedpushprod},  the fact that $\square$ preserves colimits in both variables, and the fact that $I$ is sifted. 
\end{proof}
\end{proposition}

\begin{proposition}\label{idk2}
Let $\cC$ be a bicomplete category,
let $\alpha \colon G \to H$ be a group homomorphism, and let $K$ be a subgroup of $H$. The functor $(\alpha_!(-))_{K}$ is the functor that sends a object $X$ in $\cC^G$ to the following object in $\cC$
$$
\bigsqcup_{\alpha(G)hK\in \alpha(G) \backslash H/K} X_{\alpha^{-1}(hKh^{-1})}.
$$
\begin{proof}
We compute the right adjoint of this functor. It sends an object $Y$ of $\cC$ to the image under $\alpha^*$ of $\hom (H/K, Y)$. The latter object if the same as $\hom (\alpha^*(H/K), Y)$. Moreover, we have that
$$
\alpha^* (H/K)\simeq \bigsqcup_{\alpha(G)hK\in \alpha(G) \backslash H/K} G/\alpha^{-1}\left(hKh^{-1}\right). 
$$
We deduce the right adjoint of the functor we care about sends an object $Y$ of $\cC$ to 
$$
\prod_{\alpha(G)hK\in \alpha(G) \backslash H/K} \hom \left(G/\alpha^{-1}\left(hKh^{-1}\right), Y \right)
$$
which implies that the image of a $G$-object $X$ in $\cC$ under $(\alpha_!(X))_K$ is
$$
\bigsqcup_{\alpha(G)hK\in \alpha(G) \backslash H/K} X_{\alpha^{-1}(hKh^{-1})}.
$$
\end{proof}
\end{proposition}

\begin{proposition}\label{idk}
Let $\cC$ be a cocomplete and closed monoidal category. For every group morphism $\alpha\colon G \to H$, for every $X$ in $\cC^G$, and for every $Y$ in $\cC^H$, there is an isomorphism

\begin{equation}\label{equidk}
\alpha_!(X\otimes \alpha^*Y)\simeq  \alpha_!X\otimes Y.  
\end{equation}

\begin{proof}
For every $Y$ in $\cC^H$, both sides of \ref{equidk} give functors $\cC^G \to \cC^H$ and we compute their right adjoints. For the right-hand side, the right adjoint is given by the functor that sends an object $Z$ in $\cC^H$ to 
$\alpha^*\hom(Y,Z)$. For the left-hand side, the right adjoint is the functor that sends $Z$ to $\hom(\alpha^*Y, \alpha^*Z)$. The result follows from the fact that there is a natural isomorphism 
$$
\hom(\alpha^*Y, \alpha^*Z)\simeq \alpha^*\hom(Y,Z).
$$
\end{proof}
\end{proposition}

\begin{proposition}
For every strong fixed cofibrant operad $\cO$ in $\fun(I,\sSet)^G$, for every strong fixed cofibration $u\colon A \to B$ in $\coll{}(\fun(I,\sSet)^G)$ and for every pushout square
$$
\begin{tikzcd}
F(A) \arrow[r] \arrow[d,"F(u)"'] & \cO \arrow[d,"f"] \\
F(B) \arrow[r] & \cP,
\end{tikzcd}
$$
the morphism $U(f)$ is a strong fixed cofibration. In particular, $\cP$ is also strong fixed cofibrant.
\begin{proof}
We recall first from \cite[Proposition 5.2]{PS1} the construction of the $\Sigma_n$-morphism $f(n)\colon \cO(n) \to \cP(n)$. The morphism $f(n)$ is obtained as the transfinite composition
$$
\cO=\cP_0(n) \overset{f_0(n)}{\longrightarrow} \cP_1(n) \overset{f_1(n)}{\longrightarrow} \cP_{2}(n) \overset{f_{2}(n)}{\longrightarrow} \cdots
$$
where every $f_k(n)$ is defined by induction as a pushout
$$
\begin{tikzcd}
\bigsqcup_{T} \Sigma_n \times_{\operatorname{Aut}(T)} x^*(T) \arrow[r] \arrow[d,"\bigsqcup_{T} \Sigma_n \times_{\operatorname{Aut}(T)} \epsilon(T)"'] & \cP_k(n) \arrow[d,"f_k(n)"] \\ 
\bigsqcup_{T} \Sigma_n \times_{\operatorname{Aut}(T)} x(T) \arrow[r] & \cP_{k+1}(n)
\end{tikzcd}
$$
where the coproducts run over all isomorphism classes of admissible marked trees with $n$ leaves and $k$ marked vertices, and where the $\operatorname{Aut}(T)$-morphism $\epsilon(T)\colon x^*(T) \to x(T)$ is inductively defined as 
$$
\epsilon(T):= \epsilon(l)\square \bigpp_{i=1}^l \epsilon(T_i)
$$
with $l$ the arity of the root of $T$ and with $\epsilon(l) \in \fun(I,\sSet)^{G\times \Sigma_{l}}$ defined as 
$$ \epsilon(l)=
\begin{cases}
u(l)   & \text{if $l$ is marked}; \\
\eta_{\cO}(l)  & \text{if $l$ is not marked},
\end{cases}
$$
with $\eta_{\cO}\colon 1_{\cC} \to \cO$ the morphism given by the identity of $\cO$.

Using \ref{stabcof}, it is enough to prove  that $\Sigma_n \times_{\operatorname{Aut}(T)} \epsilon(T)$ is a strong fixed cofibration for every marked tree $T$.
Let $\sP=(U_1,\ldots, U_j)$ be a partition of $\{1,\ldots, n\}$ with $|\sP_i|=n_i$. For every tuple $(f_1,\ldots,f_j)$ of fixed cofibrations, we have to prove that
$$
\left (\alpha_!\left(\epsilon(T)\right) \square \bigpp_{\sP} f \right)_{\Sigma_{\sP}}
$$
is a fixed cofibration. Using \ref{idk2}, \ref{idk} and the fact that the conjugation of a subgroup of the form $\Sigma_{\sP}$ in $\Sigma_n$ is also of this form, it is enough to prove that
$$
\left(\epsilon(T) \square \bigpp_{\sP} f\right)_{\operatorname{Aut}_{\sP}(T)}
$$
is a fixed cofibration, where the $\operatorname{Aut}_{\sP}(T)$-action on $\bigpp_{\sP} f$ is given by precomposing with the homomorphism $\operatorname{Aut}_{\sP}(T) \to \Sigma_{\sP}$. We will prove the result by grafting induction on $T$ as a partitioned tree. Suppose $T=t_m(T_1,\ldots, T_m)$ such that we are in the situation of \ref{decomptree}.
We have 
$$
\epsilon(T)=\epsilon(m)\pp \bigpp_{i=1}^k \epsilon(T_i)^{\square^{m_i}}
$$
and we can write $\epsilon(T) \square \bigpp_{\sP} f$ as
$$
\epsilon(m)\pp \bigpp_{i=1}^k  \left( \epsilon(T_i) \square \bigpp_{\sP_i}f \right)^{\square^{m_i}}.
$$
We deduce that
$$
\left( \epsilon(T) \square \bigpp_{\sP}f \right)_{\operatorname{Aut}_{\sP}(T)} \simeq 
\left( \epsilon(m)\pp \bigpp_{i=1}^k  \left( \left( \epsilon(T_i) \square \bigpp_{\sP_i}f \right)_{\operatorname{Aut}_{\sP}(T_i)}\right)^{\square^{m_i}} \right)_{\prod_{i=1}^k \Sigma_{m_i}},
$$
using the description of $\operatorname{Aut}_{\sP}(T)$ of \ref{decomptree}.
Finally, we deduce that $\left( \epsilon(T) \square \bigpp_{\sP}f \right)_{\operatorname{Aut}_{\sP}(T)}$ is a fixed cofibration because $\epsilon(m)$ is a strong fixed cofibration and because $ \left( \epsilon(T_i) \square \bigpp_{\sP_i}f \right)_{\operatorname{Aut}(T_i)}$ is a fixed cofibration by induction.
\end{proof}
\end{proposition}

\begin{proposition}\label{precof}
For every strong fixed cofibrant operad $\cO$ in $\fun(I,\sSet)^G$ and for every subgroup $H$ of $G$, the $H$-fixed points functor $(-)^H$ given by the composition
$$
\Algg_{\cO}(\fun(I,\sSet)^G) \overset{U}{\longrightarrow} \fun(I,\sSet)^G \overset{(-)^H}{\longrightarrow} \fun(I,\sSet)
$$
preserves cofibrant objects. Furthermore, for every cellular algebra $F$ in $\Algg_{\cO}(\fun(I,\sSet)^G)$, the colimit of $U(F)\colon I \to \sSet^G$ commutes with $H$-fixed points.
\begin{proof}
Let $\cA$ be a cellular algebra over $\cO$. Using \cite[Proposition 5.4]{BM03}, there is a morphism of operads $\cO \to \cO [A]$ such that $\cO [A](0)=A$. Moreover, $\cO \to \cO [A]$ is a sequential colimit of cellular extension by strong fixed cofibrations and it follows that $\cO \to \cO [A]$ is a strong fixed cofibration. In particular, because $\cO$ is strong fixed cofibrant, $\cO[A]$ is also strong fixed cofibrant. Finally, because $A=\cO [A](0)$, we have that $A^H$ is cofibrant for every subgroup $H$ of $G$ and the colimit of $A$ commutes with $H$-fixed point. Finally, if $B$ is a retract of a cellular algebra $A$, then $B^H$ is a retract of $A^H$ and it follows that $B^H$ is cofibrant.
\end{proof}
\end{proposition}

\begin{proposition}\label{proppresifhoco}
For every operad $\cO$ in $\sSet^G$ such that the $\Sigma_n$-action on $\cO_n$ is free, the fixed point functor 
$$
\Algg_{\cO}(\sSet^G) \overset{U}{\longrightarrow} \sSet^G \overset{(-)^H}{\longrightarrow} \sSet
$$
preserves sifted homotopy colimits.
\begin{proof}
Suppose that $I$ is a homotopy sifted category (in particular, $I$ is sifted). As pointed out in \cite[Proposition 7.9]{PS1}, we can compute all sifted homotopy colimits by only considering homotopy sifted categories with finite coproducts. We can thus assume that $I$ admits coproducts. For every $n$, using \cite{STE}, $\cO_n$ is a cellular object in $\sSet^{G\times \Sigma_n}$ for the graph model structure (and for n=1, the identity map $* \to \cO(1)$ is also cellular). In particular, $\cO_n$ is also cellular when seen as a constant functor in $\fun (I,\sSet)^{G\times \Sigma_n}$ and this implies that $\cO$ is strong fixed cofibrant as an operad in $\fun(I,\sSet)^G$ using \ref{propcelfixcof}.
Finally, we have an equivalence $\Algg_{\cO}(\fun(I,\sSet)^G)\simeq \fun (I,\Algg_{\cO}(\sSet^G))$ and the model structure on $\Algg_{\cO}(\fun(I,\sSet)^G)$ corresponds to the projective model structure on $\fun (I,\Algg_{\cO}(\sSet^G))$. Consider now $F\colon I \to \Algg_{\cO}(\sSet^G)$ a cellular functor. As $F$ is strong fixed cofibrant, the functor $U(F)^H$ is cofibrant and $U(-)^H$ preserves the colimit of $F$ because $U$ preserves sifted colimits and because $(-)^H$ preserves the colimit of $U(F)$. The result now follows from the fact that homotopy colimits can always be computed as strict colimits of cellular functors.
\end{proof}
\end{proposition}

In particular, we obtain the following result on the associated $\infty$-categories, where we denote by $\mathbf{Alg}_{\cO}$ the $\infty$-category obtained from $\Algg_{\cO}(\sSet)$ by inverting $G$-weak homotopy equivalences.
\begin{corollary}\label{nsifco}
For every operad $\cO$ in $\sSet^G$ such that the $\Sigma_n$-action on $\cO_n$ is free, the forgetful functor
$$
U\colon \mathbf{Alg}_{\cO} \to \sS^G
$$
preserves sifted colimits.
\begin{proof}
Using that colimits in a presheaf category can be computed component-wise, it is enough to show that the composition 
$$
U\colon \mathbf{Alg}_{\cO} \overset{U}{\longrightarrow} \sS^G \overset{\operatorname{ev}_{G/H}}{\longrightarrow} \sS
$$
preserves sifted colimits for every subgroup $H\subseteq G$.
This functor is the right derived functor of the composition 
$$
\Algg_{\cO}(\sSet^G) \overset{U}{\longrightarrow} \sSet^G \overset{(-)^H}{\longrightarrow} \sSet
$$
and, combining \cite[Proposition 5.3.1.16]{MR2522659}, \cite[Corollary 5.5.8.17]{MR2522659} and \cite[Theorem 1.3.4.24]{HA} with \ref{proppresifhoco}, the functor $U:\mathbf{Alg}_{\cO} \to \sS^G$ preserves sifted colimits.
\end{proof}
\end{corollary}

The previous result applies in particular to algebras over an $N_{\infty}$-operad.

\begin{corollary}
For every $N_{\infty}$-operad $\cO$ in $\sSet^G$, the forgetful functor
$$
U\colon \mathbf{Alg}_{\cO} \to \sS^G
$$
preserves sifted colimits.
\end{corollary}
\bibliographystyle{amsalpha}
\bibliography{references}

\end{document}